\documentclass[11pt,a4paper]{article}

\usepackage[margin=2cm]{geometry}

\usepackage{amsmath,amssymb,amsthm}
\usepackage[numbers]{natbib}

\usepackage[sans]{dsfont}
\newcommand{\unito}{{\mathds 1}}

\let\eps\varepsilon

\DeclareMathOperator\ev{ev}
\DeclareMathOperator\id{id}
\DeclareMathOperator\Id{Id}
\DeclareMathOperator\inl{inl}
\DeclareMathOperator\inr{inr}
\DeclareMathOperator\Ob{Ob}
\DeclareMathOperator\op{op}
\DeclareMathOperator\pr{pr}

\DeclareMathOperator\C{\mathcal{C}}
\DeclareMathOperator\D{\mathcal{D}}

\DeclareMathOperator\Set{\mathbf{Set}}

\usepackage[all,dvips]{xy}
\SelectTips{cm}{10}

\IfFileExists{url.sty}{\usepackage{url}}{}
\providecommand{\url}[1]{{\tt #1}}

\usepackage{ifpdf}
\ifpdf
    \IfFileExists{hyperref.sty}{\usepackage[pdftex]{hyperref}}{}
\else
    \IfFileExists{hyperref.sty}{\usepackage[hypertex]{hyperref}}{}
\fi

\swapnumbers

\newtheorem{proposition}[subsection]{Proposition}

\theoremstyle{definition}
\newtheorem{remark}[subsection]{Remark}

\newtheorem{definition}[subsection]{Definition}

\numberwithin{equation}{section}

\newcommand{\proref}[1]{Proposition~\ref{#1}}

\newcommand{\diaref}[1]{Diagram~\ref{#1}}

\title{Calculating monad transformers with category theory}
\author{Oleksandr Manzyuk}

\begin{document}
\maketitle

\begin{abstract}
  We show that state, reader, writer, and error monad transformers are
  instances of one general categorical construction: translation of
  a monad along an adjunction.
\end{abstract}

\section{Introduction}

This note is an elaboration of \cite{calc-mon-cat-th} (hence the
title).  The latter serves as a very gentle introduction to category
theory for Haskell programmers.  In particular, it explains how monads
arise from adjunctions between categories.  We take this (arguably
well-known in the category theory community) idea one step further.
We show that monads can be translated along adjunctions and illustrate
by examples how the standard monad transformers --- state, reader,
writer, error --- can be interpreted as instances of this
construction.  Notably, continuation monad transformers do not fit
into this framework.

Unlike \cite{calc-mon-cat-th}, this note gives more details, and is as
a consequence more technical.  It makes greater emphasis on category
theory rather than on programming.  Like in \cite{calc-mon-cat-th},
the reader should consider unproven or partially proven statements to
be exercises.

\section{Translating monads along adjunctions}

We begin by recalling one of the fundamental notions of category
theory: the notion of adjunction.  The reader is referred to
\cite{calc-mon-cat-th} and \cite[Chapter~IV]{MR1712872} for more
details.

\begin{definition}
  Let $\C$ and $\D$ be categories.  An \emph{adjunction} from $\C$ to
  $\D$ is a triple $(F, U, \varphi)$, where $F : \C \to \D$ and $U :
  \D \to \C$ are functors, and $\varphi$ is a family of bijections
  \[
  \varphi_{X,Y} : \D(FX, Y) \xrightarrow{\sim} \C(X, UY), \quad
  X\in\Ob\C, \quad Y\in\Ob\D,
  \]
  natural in $X$ and $Y$.  We say that $F$ is \emph{left adjoint} to
  $U$ and $U$ is \emph{right adjoint} to $F$.
\end{definition}

An adjunction $(F, U, \varphi) : \C \to \D$ determines two natural
transformations.  Namely, for a fixed object $X\in\Ob\C$, the family
of functions $\varphi_{X,-}=\{\varphi_{X, Y}\}_{Y\in\Ob\D}$ is a
natural transformation from the representable functor $\D(FX, -) : \D
\to \Set$ to the functor $\C(X, U(-)) : \D \to \Set$, which by the
Yoneda Lemma is given by
\begin{equation}
\varphi_{X,Y}(f) = U(f)\circ\eta_X, \quad f \in \D(FX, Y),
\label{eq-varphi}
\end{equation}
where $\eta_X = \varphi_{X, FX}(\id_{FX}) : X \to UFX$.  Clearly, the
family of morphisms $\eta_X$ is natural in $X$, giving rise to a
natural transformation $\eta : \Id_{\C} \to UF$, called the
\emph{unit} of the adjunction $(F, U, \varphi)$.  Similarly, for a
fixed object $Y\in\Ob\D$, the family of functions $\varphi^{-1}_{-, Y}
= \{\varphi^{-1}_{X,Y}\}_{X\in\Ob\C}$ is a natural transformation from
the representable functor $\C(-, UY) : \C^{\op} \to \Set$ to the
functor $\D(F(-), Y) : \C^{\op} \to \Set$, which by the Yoneda Lemma
is given by
\begin{equation}
\varphi^{-1}_{X,Y}(g) = \eps_Y \circ F(g), \quad g \in \C(X, UY),
\label{eq-varphi-inv}
\end{equation}
where $\eps_Y = \varphi^{-1}_{UY, Y}(\id_{UY}): FUY \to Y$.  The
family of morphisms $\eps_Y$ is natural in $Y$, giving rise to a
natural transformation $\eps : FU \to \Id_{\D}$, called the
\emph{counit} of the adjunction $(F, U, \varphi)$.  The identities
$\varphi^{-1}_{X, FX}(\eta_X) = \id_{FX}$ and $\varphi_{UY, Y}(\eps_Y)
= \id_{UY}$ translate into so called \emph{triangular} or
\emph{counit-unit} equations:
\begin{align}
  \Bigl[FX \xrightarrow{F(\eta_X)} FUFX \xrightarrow{\eps_{FX}} FX
  \Bigr] & = \id_{FX}, \label{eq-eps-F-eta-id}\\
  \Bigl[\,UY \xrightarrow{\eta_{UY}} UFUY \xrightarrow{\eps_{U(\eps_Y)}}
  UY \Bigr] & = \id_{UY}. \label{eq-U-eps-eta-id}
\end{align}
Conversely, if $\eta : \Id_{\C} \to UF$ and $\eps : FU \to \Id_{\D}$
are natural transformations satisfying equations
\eqref{eq-eps-F-eta-id} and \eqref{eq-U-eps-eta-id}, then the family
of functions $\varphi_{X,Y} : \D(FX, Y) \to \C(X, UY)$ given by
\eqref{eq-varphi} is natural in $X$ and $Y$, and each $\varphi_{X,Y}$
is invertible with the inverse given by \eqref{eq-varphi-inv}.
Therefore, an adjunction can equivalently be described as a quadruple
$(F, U, \eta, \eps)$, where $F : \C \to \D$ and $U : \D \to \C$ are
functors and $\eta : \Id_{\C} \to UF$ and $\eps : FU \to \Id_{\D}$ are
natural transformations subject to equations \eqref{eq-eps-F-eta-id}
and \eqref{eq-U-eps-eta-id}.  We refer the interested reader to
\cite[Chapter~IV, Theorem~2]{MR1712872} for more equivalent
definitions of an adjunction.

Every adjunction $(F, U, \eta, \eps)$ from $\C$ to $\D$ gives rise to
a monad $(P, e^P, m^P)$ on $\C$, where $P = UF$, $e^P_X = \eta_X : X
\to UFX$, and $m^P_X = U(\eps_{FX}) : UFUFX \to UFX$.  The following
proposition is a mild generalization of this observation.  The latter
can be recovered by taking $T$ to be the identity monad.

\begin{proposition}\label{pro-trans-monad-along-adjunction}
  Let $(F, U, \eta, \eps)$ be an adjunction from $\C$ to $\D$.
  Suppose that $(T, e^T, m^T)$ is a monad on the category $\D$.  Then
  the functor $P = UTF : \C \to \C$ equipped with the natural
  transformations
  \begin{align}
    e^P_X & = \Bigl[ X \xrightarrow{\eta_X} UFX \xrightarrow{U(e^T_{FX})}
    UTFX \Bigr],
    \\
    m^P_X & = \Bigl[ UTFUTFX \xrightarrow{UT(\eps_{TFX})} UTTFX
    \xrightarrow{U(m^T_{FX})} UTFX \Bigr]
  \end{align}
  is a monad on the category $\C$.
\end{proposition}

\begin{proof}
  Let us check the monad axioms.  The identity axioms are proven in
  \diaref{dia-identity-axioms-P}.
  \begin{figure}
  \[
  \begin{xy}
    (-55, 24)*+{UTFX}="1";
    (0, 24)*+{UFUTFX}="2";
    (55, 24)*+{UTFUTFX}="3";
    (0, 0)*+{UTFX}="4";
    (55, 0)*+{UTTFX}="5";
    (55, -24)*+{UTFX}="6";
    (-55, 0)*+{UTFUFX}="7";
    (-55, -24)*+{UTFUTFX}="8";
    (0, -24)*+{UTTFX}="9";
    {\ar@{->}^-{\eta_{UTFX}} "1";"2"};
    {\ar@{->}^-{U(e^T_{FUTFX})} "2";"3"};
    {\ar@{=} "1";"4"};
    {\ar@{->}^-{U(\eps_{TFX})} "2";"4"};
    {\ar@{->}^-{UT(\eps_{TFX})} "3";"5"};
    {\ar@{->}^-{U(e^T_{TFX})} "4";"5"};
    {\ar@{=} "4";"6"};
    {\ar@{->}^-{U(m^T_{FX})} "5";"6"};
    {\ar@{->}_-{UTF(\eta_X)} "1";"7"};
    {\ar@{->}^-{UT(\eps_{FX})} "7";"4"};
    {\ar@{->}_-{UTFU(e^T_{FX})} "7";"8"};
    {\ar@{->}^-{UT(\eps_{TFX})} "8";"9"};
    {\ar@{->}^-{U(m^T_{FX})} "9";"6"};
    {\ar@{->}_-{UT(e^T_{FX})} "4";"9"};
  \end{xy}
  \]
  \caption{Proof of the identity axioms for the monad $P$.}
  \label{dia-identity-axioms-P}
  \end{figure}
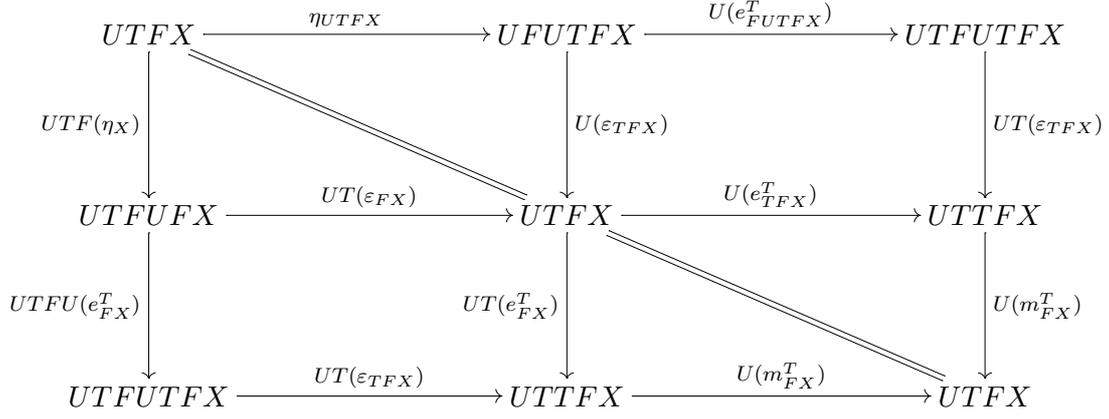
  Each cell of this diagram commutes.  The pair of top left triangles
  commute by the counit-unit equations \eqref{eq-eps-F-eta-id} and
  \eqref{eq-U-eps-eta-id}.  The pair of bottom right triangles commute
  by the identity axioms for the monad $T$.  The commutativity of the
  top right square follows from the naturality of $e^T$, while the
  commutativity of the bottom left square follows from the naturality
  of $\eps$.  The associativity axiom coincides with the exterior of
  \diaref{dia-assoc-axiom-P}.  The commutativity of the two squares on
  the left follows from the naturality of $\eps$.  The top right
  square commutes by the naturality of $m^T$.  Finally, the bottom
  right square commutes by the associativity axiom for the monad $T$.
  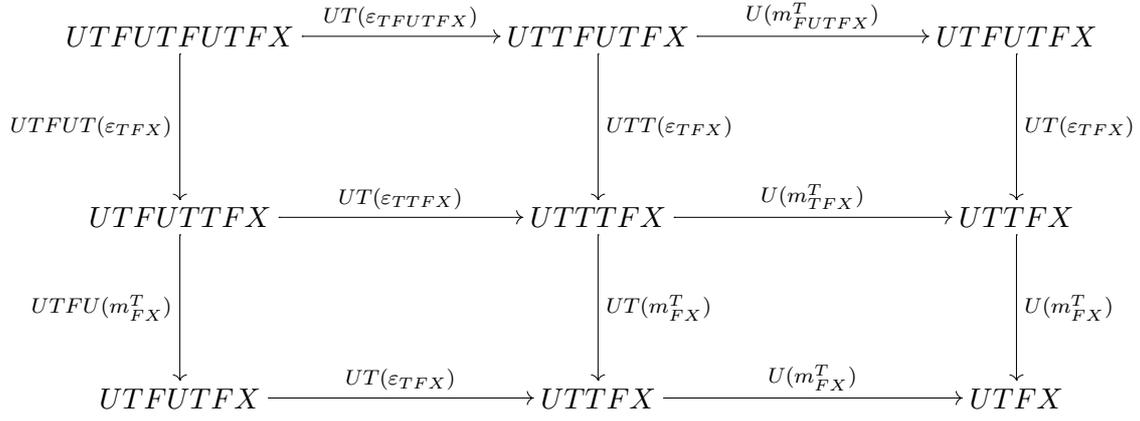
\begin{figure}
  \[
  \begin{xy}
    (-55, 24)*+{UTFUTFUTFX}="1";
    (0, 24)*+{UTTFUTFX}="2";
    (55, 24)*+{UTFUTFX}="3";
    (-55, 0)*+{UTFUTTFX}="4";
    (0, 0)*+{UTTTFX}="5";
    (55, 0)*+{UTTFX}="6";
    (-55, -24)*+{UTFUTFX}="7";
    (0, -24)*+{UTTFX}="8";
    (55, -24)*+{UTFX}="9";
    {\ar@{->}^-{UT(\eps_{TFUTFX})} "1";"2"};
    {\ar@{->}^-{U(m^T_{FUTFX})} "2";"3"};
    {\ar@{->}_-{UTFUT(\eps_{TFX})} "1";"4"};
    {\ar@{->}^-{UTT(\eps_{TFX})} "2";"5"};
    {\ar@{->}^-{UT(\eps_{TFX})} "3";"6"};
    {\ar@{->}^-{UT(\eps_{TTFX})} "4";"5"};
    {\ar@{->}^-{U(m^T_{TFX})} "5";"6"};
    {\ar@{->}_-{UTFU(m^T_{FX})} "4";"7"};
    {\ar@{->}^-{UT(m^T_{FX})} "5";"8"};
    {\ar@{->}^-{U(m^T_{FX})} "6";"9"};
    {\ar@{->}^-{UT(\eps_{TFX})} "7";"8"};
    {\ar@{->}^-{U(m^T_{FX})} "8";"9"};
  \end{xy}
  \]
  \caption{Proof of the associativity axiom for the monad $P$.}
  \label{dia-assoc-axiom-P}
  \end{figure}
\end{proof}

\begin{remark}
  \proref{pro-trans-monad-along-adjunction} allows us to translate a
  monad on the category $\D$ into a monad on the category $\C$.  In
  functional programming and denotational semantics, we are primarily
  interested in monads that are strong
  \cite[Definition~3.2]{MR1115262}.  We recall that a monad $(T, e^T,
  m^T)$ on a cartesian category $\C$ is \emph{strong} if it is
  equipped with a transformation $t_{X,Y} : TX \times Y \to T(X\times
  Y)$ natural in $X$ and $Y$ and compatible with both the cartesian
  and monad structures; see \cite[Definition~3.2]{MR1115262} for the
  precise compatibility conditions.  Under what conditions is the
  translation of a strong monad along an adjunction also a strong monad?
  I don't have a good answer.  In each of the examples we consider
  below, this question can be resolved in an ad hoc manner.  However,
  I am not aware of a general condition that applies to all the
  examples.  That is why I am going to ignore this issue henceforth.
  Note that in two out of the four examples we will need the
  assumption that the monad being translated is strong.
\end{remark}

\section{State monad transformer}

For the sake of simplicity, we assume that the categories $\C$ and
$\D$ are the category $\Set$ of sets.  An arbitrary set $S$ gives rise
to an adjunction $(F, U, \eta, \eps)$, where $F = - \times S$, $U =
(-)^S$, $\eta_X : X \to (X\times S)^S$ is given by $\eta_X(x) =
\lambda s. (x, s)$, and $\eps_X : X^S\times S\to X$ is given by
$\eps_X(f, s) = f(s)$.  The monad $(P, e^P, m^P)$ associated with this
adjunction is the state monad with the state $S$: $PX = (X\times
S)^S$, $e^P_X(x) = \eta_X(x) = \lambda s.(x, s)$, and $m^P_X(g) =
\eps_{FX}\circ g = \lambda s. \textup{ \textbf{let} } (f, s') = g(s)
\textup{ \textbf{in} } f(s')$.

More generally, suppose $(T, e^T, m^T)$ is a monad on $\Set$.  Let us
compute explicitly the monad $(P, e^P, m^P)$ obtained by translating
$T$ along the adjunction $(F, U, \eta, \eps)$.  We have: $PX =
(T(X\times S))^S$, $e^P_X(x) = e^T_{FX} \circ \eta_X(x) = \lambda
s. e^T_{X\times S}(x, s)$, $m^P_X(g) = m^T_{FX} \circ
T(\eps_{TFX})\circ g$.  Translating the last equation into Haskell
notation, we obtain:
\begin{verbatim}
join :: Monad m => (s -> m (s -> m (a, s), s)) -> s -> m (a, s)
join g = join . fmap ev . g where ev (f, s') = f s'
\end{verbatim}
which can be transformed as follows:
\begin{verbatim}
join g = \s -> join $ fmap ev $ g s

       -- definitions of `fmap' and `join'
       = \s -> (g s >>= (return . ev)) >>= id

       -- associativity axiom for monads
       = \s -> g s >>= (\(f, s') -> return (ev (f, s')) >>= id)

       -- left identity axiom for monads
       = \s -> g s >>= (\(f, s') -> f s')

       -- syntactic sugar
       = \s -> do (f, s') <- g s; f s'
\end{verbatim}
Modulo some newtype constructor wrapping/unwrapping this is precisely
the multiplication in the monad \verb|StateT s m|.

\section{Writer monad transformer}

Let $\C$ be the category $\Set$ of sets.  Let $M$ be a monoid with the
binary operation $\cdot : M\times M\to M$, $(m_1, m_2)\mapsto m_1\cdot
m_2$, and the neutral element $1\in M$.  Let $\D$ be the category
$M\textup{-}\Set$ of $M$-sets: objects are $M$-sets, i.e., pairs $(X,
a)$, where $X$ is a set and $a : X\times M \to X$, $(x, m) \mapsto
x^m$ is an action morphism such that $x^1 = x$ and $(x^{m_1})^{m_2} =
x^{m_1 \cdot m_2}$, for all $x\in X$ and $m_1, m_2\in M$, and
morphisms are $M$-equivariant maps, i.e., a morphism $f : (X, a) \to
(Y, b)$ is a map $f : X \to Y$ such that $f(x^m) = f(x)^m$, for all
$x\in X$ and $m\in M$.  There is an adjunction $(F, U, \eta, \eps)$
from $\Set$ to $M\textup{-}\Set$.  The functor $F : \Set \to M
\textup{-} \Set$ maps a set $X$ to the free $M$-set $X\times M$ with
the action $(X\times M)\times M\to X\times M$ given by $(x, m)^n = (x,
m\cdot n)$.  The functor $U : M \textup{-} \Set \to \Set$ is the
forgetful functor that maps an $M$-set $(X, a)$ to its underlying set
$X$.  The unit $\eta_X : X \to X\times M$ is given by $\eta_X(x) = (x,
1)$, and the counit $\eps_{(X, a)} : X\times M\to X$ is simply the
action morphism $a$.  The monad associated with this adjunction is the
writer monad with the monoid $M$.

We conclude by \proref{pro-trans-monad-along-adjunction} that a monad
$(T, e^T, m^T)$ on the category $\D = M \textup{-} \Set$ gives rise to
a monad on $\C = \Set$.  However, this is not what we would like to
have: we would like to produce a monad on $\Set$ out of another monad
$(T, e^T, m^T)$ on $\Set$, not on $M \textup{-} \Set$.  This is
possible if the monad $T$ is strong.  Let $t_{X, Y}: TX \times Y \to T
(X\times Y)$ be the strength of the monad $T$.  Then $(T, e^T, m^T)$
induces a monad $(\bar{T}, e^{\bar T}, m^{\bar T})$ on the category $M
\textup{-} \Set$.  Namely, if $(X, a)$ is an $M$-set, then the set
$TX$ becomes an $M$-set if we equip it with the action
\begin{equation}\label{eq-induced-action}
b = \Bigl[ TX\times M \xrightarrow{t_{X,M}} T(X\times M)
\xrightarrow{T(a)} TX \Bigr].
\end{equation}
Let us prove that $b$ is an action, i.e., that it satisfies the
identity and associativity conditions.  It is convenient to first
express these conditions diagrammatically.  A map $a : X \times M\to
M$ is an action if it satisfies the identity axiom:
\[
\Bigl[ X\xrightarrow[\sim]{\rho_X}X\times\unito \xrightarrow{\id_X
  \times 1_M} X\times M \xrightarrow{a} X \Bigr] = \id_X,
\]
and the associativity axiom:
\[
\begin{xy}
  (-20, 18)*+{(X\times M)\times M}="1";
  (20, 18)*+{X\times M}="2";
  (-20, 0)*+{X\times (M\times M)}="3";
  (-20, -18)*+{X\times M}="4";
  (20, -18)*+{X}="5";
  {\ar@{->}^-{a\times\id_M} "1";"2"};
  {\ar@{->}^-{a} "2";"5"};
  {\ar@{->}_-{\alpha_{X, M, M}}^-{\wr} "1";"3"};
  {\ar@{->}_-{\id_X\times\cdot} "3";"4"};
  {\ar@{->}^-{a} "4";"5"};
\end{xy}
\]
Here $\rho$ and $\alpha$ are the right unit and associativity
constraints of the monoidal structure induced by the cartesian
product, $\unito = \{*\}$ is the terminal object (a singleton), and
$1_M : \unito \to M$, $*\mapsto 1$.  Suppose that $a : X\times M \to
X$ is an action.  Let us prove that the map $b$ given by
\eqref{eq-induced-action} is also an action.  The identity axiom
is proven in the following diagram:
\[
\begin{xy}
  (-20, 18)*+{TX\times\unito}="1";
  (20, 18)*+{TX\times M}="2";
  (-20, 0)*+{T(X\times\unito)}="3";
  (20, 0)*+{T(X\times M)}="4";
  (-20, -18)*+{TX}="5";
  (20, -18)*+{TX}="6";
  {\ar@{->}^-{\id_{TX}\times 1_M} "1";"2"};
  {\ar@{->}^-{t_{X,\unito}} "1";"3"};
  {\ar@{->}_-{t_{X,M}} "2";"4"};
  {\ar@{->}^-{T(\id_X\times 1_M)} "3";"4"};
  {\ar@{->}_-{T(\rho_X)} "5";"3"};
  {\ar@{->}_-{T(a)} "4";"6"};
  {\ar@{=} "5";"6"};
  {\ar@{->}@/^3pc/^-{\rho_{TX}} "5";"1"};
  {\ar@{->}@/^3pc/^-{b} "2";"6"};
\end{xy}
\]
The top square commutes by the naturality of the strength $t$.  The
bottom square commutes by the identity axiom for the action $a$ (and
functoriality of $T$).  The left triangle is one of the strength
axioms, and the right triangle is the definition of $b$.  The
associativity axiom coincides with the exterior of the following
diagram:
\[
\begin{xy}
  (-45, 27)*+{(TX\times M)\times M}="1";
  (0, 27)*+{T(X\times M)\times M}="2";
  (45, 27)*+{TX\times M}="3";
  (0, 9)*+{T((X\times M)\times M)}="4";
  (45, 9)*+{T(X\times M)}="5";
  (-45, -9)*+{TX\times (M\times M)}="6";
  (0, -9)*+{T(X\times (M\times M))}="7";
  (-45, -27)*+{TX\times M}="8";
  (0, -27)*+{T(X\times M)}="9";
  (45, -27)*+{TX}="10";
  {\ar@{->}^-{t_{X, M}\times\id_M} "1";"2"};
  {\ar@{->}^-{T(a)\times \id_M} "2";"3"};
  {\ar@{->}_-{\alpha_{TX,M,M}} "1";"6"};
  {\ar@{->}^-{t_{X\times M, M}} "2";"4"};
  {\ar@{->}^-{t_{X, M}} "3";"5"};
  {\ar@{->}^-{T(a\times\id_M)} "4";"5"};
  {\ar@{->}^-{T(\alpha_{X,M,M})} "4";"7"};
  {\ar@{->}^-{T(a)} "5";"10"};
  {\ar@{->}^-{t_{X,M\times M}} "6";"7"};
  {\ar@{->}^-{\id_{TX}\times\cdot} "6";"8"};
  {\ar@{->}^-{T(\id_X\times\cdot)} "7";"9"};
  {\ar@{->}^-{t_{X,M}} "8";"9"};
  {\ar@{->}^-{T(a)} "9";"10"};
\end{xy}
\]
The left pentagon is one of the strength axioms.  The right pentagon
commutes by the associativity condition for $a$.  The remaining
squares commute by the naturality of $t$.

We have proven that once $(X, a)$ is an $M$-set, the pair $(TX, b)$,
where $b$ is given by \eqref{eq-induced-action}, is also an $M$-set.
We set $\bar{T}(X, a) = (TX, b)$.  Let us check that if $f : (X, a)
\to (X', a')$ is an $M$-equivariant map, then $T(f) : TX \to TX'$ is
actually an $M$-equivariant map $\bar{T}(X, a) \to \bar{T}(X', a')$.
This is easy: the equivariance of $f$ is expressed by the
commutativity of the diagram
\[
\begin{xy}
  (-15, 9)*+{X\times M}="1";
  (15, 9)*+{X'\times M}="2";
  (-15, -9)*+{X}="3";
  (15, -9)*+{X'}="4";
  {\ar@{->}^-{f\times\id_M} "1";"2"};
  {\ar@{->}_-{a} "1";"3"};
  {\ar@{->}^-{a'} "2";"4"};
  {\ar@{->}^-{f} "3";"4"};
\end{xy}
\]
The naturality of $t$ and the functoriality of $T$ imply that the
following diagram commutes, too:
\[
\begin{xy}
  (-20, 27)*+{TX\times M}="1";
  (20, 27)*+{TX'\times M}="2";
  (-20, 9)*+{T(X\times M)}="3";
  (20, 9)*+{T(X'\times M)}="4";
  (-20, -9)*+{TX}="5";
  (20, -9)*+{TX'}="6";
  {\ar@{->}^-{Tf\times\id_M} "1";"2"};
  {\ar@{->}_-{t_{X, M}} "1";"3"};
  {\ar@{->}^-{t_{X', M}} "2";"4"};
  {\ar@{->}^-{T(f\times\id_M)} "3";"4"};
  {\ar@{->}_-{T(a)} "3";"5"};
  {\ar@{->}^-{T(a')} "4";"6"};
  {\ar@{->}^-{T(f)} "5";"6"};
\end{xy}
\]
The vertical compositions are precisely the action morphisms of
$\bar{T}(X, a)$ and $\bar{T}(X', a')$.  Therefore, the above diagram
expresses the fact that the map $T(f) : TX\to TX'$ is indeed an
$M$-equivariant map $\bar{T}(X, a)\to \bar{T}(X', a')$.  Hence, we can
set $\bar{T}(f) = T(f)$.  Then $\bar{T}$ is a functor from the
category $M \textup{-} \Set$ to itself.  The functoriality of
$\bar{T}$ follows immediately from that of $T$.

One can also prove that the natural transformations $e^T_X : X\to TX$
and $m^T_X : TTX\to TX$ induce natural transformations
$e^{\bar{T}}_{(X, a)}: (X, a)\to \bar{T}(X, a)$ and $m^{\bar{T}}_{(X,
  a)} : \bar{T}\bar{T}(X, a)\to \bar{T}(X, a)$.  It suffices to check
that $e^T_X$ and $m^T_X$ are $M$-equivariant if $X$ is an $M$-set,
which follows directly from the strength axioms expressing the
compatibility of $t$ with the unit and multiplication of the monad
$T$.  For example, the equivariance of $m^T_X$ is proven in the
following diagram:
\[
\begin{xy}
  (-20, 27)*+{TTX\times M}="1";
  (20, 27)*+{TX\times M}="2";
  (-20, 9)*+{T(TX\times M)}="3";
  (-20, -9)*+{TT(X\times M)}="4";
  (20, -9)*+{T(X\times M)}="5";
  (-20, -27)*+{TTX}="6";
  (20, -27)*+{TX}="7";
  {\ar@{->}^-{m^T_X\times\id_M} "1";"2"};
  {\ar@{->}_-{t_{TX, M}} "1";"3"};
  {\ar@{->}^-{t_{X, M}} "2";"5"};
  {\ar@{->}_-{T(t_{X, M})} "3";"4"};
  {\ar@{->}^-{m^T_{X\times M}} "4";"5"};
  {\ar@{->}_-{TT(a)} "4";"6"};
  {\ar@{->}^-{T(a)} "5";"7"};
  {\ar@{->}^-{m^T_X} "6";"7"};
\end{xy}
\]
The pentagon is a strength axiom, and the square is a consequence of
the naturality of $m^T$.  A similar argument shows that $e^T_X$ is
also $M$-equivariant.

It is straightforward that the natural transformations $e^{\bar
  T}_{(X, a)}$ and $m^{\bar T}_{(X, a)}$ satisfy the monad laws,
because the underlying maps $e^T_X$ and $m^T_X$ satisfy these laws,
and because the action of $\bar T$ on morphisms coincides with that of
$T$.

Thus we have shown that a strong monad $(T, e^T, m^T)$ on the category
$\Set$ induces a monad $(\bar{T}, e^{\bar T}, m^{\bar T})$ on the
category $M\textup{-}\Set$, which can now be translated along the
adjunction $(F, U, \eta, \eps)$ by
\proref{pro-trans-monad-along-adjunction}.  Let us compute the
obtained monad $(P, e^P, m^P)$ explicitly.  The functor $P = U\bar{T}F
: \Set \to \Set$ is given by $PX = T(X\times M)$.  The unit $e^P_X : X
\to T(X\times M)$ of the monad $P$ is given by $e^P_X = e^T_{X\times
  M}\circ \eta_X = \lambda x.\; e^T_{X\times M}(x, 1)$, which is the
\verb|return| method of the writer monad transformer.  Let us compute
the multiplication.  By \proref{pro-trans-monad-along-adjunction},
$m^P_X = m^T_{X\times M}\circ T(b)$, where $b = T(a) \circ t_{X\times
  M} : T(X\times M)\times M\to T(X\times M)$ is the action morphisms
of $T(X\times M)$, and $a : (X\times M)\times M\to X\times M$, $((x,
m_1), m_2)\mapsto (x, m_1\cdot m_2)$ is the action morphism of
$X\times M$.  Translating this into Haskell notation, we obtain:
\begin{verbatim}
join = join . fmap b
    where
      b               = fmap a . strength
      a ((x, m1), m2) = (x, m1 `mappend` m2)
      strength (c, m) = c >>= (\r -> return (r, m))
\end{verbatim}
Eta-expanding this definition, we obtain:
\begin{verbatim}
join z = join (fmap b z)

       -- definitions of `join' and `fmap'
       = (z >>= (return . b)) >>= id

       -- associativity axiom for monads
       = z >>= \p -> return (b p) >>= id

       -- left identity axiom for monads
       = z >>= \p -> b p

       -- definition of `b'
       = z >>= \p -> fmap a (strength p)

       -- definition of `fmap'
       = z >>= \(c, m) -> strength (c, m) >>= (return . a)
\end{verbatim}
The expression \verb|strength (c, m) >>= (return . a)| can be further
transformed as follows:
\begin{verbatim}
strength (c, m) >>= (return . a)

    -- definition of `strength'
    = (c >>= (\r -> return (r, m))) >>= (return . a)

    -- associativity axiom for monads
    = c >>= \(x, n) -> return ((x, n), m) >>= (return . a)

    -- left identity axiom for monads
    = c >>= \(x, n) -> return $ a ((x, n), m)

    -- definition of `a'
    = c >>= \(x, n) -> return (x, n `mappend` m)
\end{verbatim}
Therefore
\begin{verbatim}
join z = z >>= \(c, m) -> c >>= \(x, n) -> return (x, n `mappend` m)

       -- syntactic sugar
       = do (c, m) <- z
            (x, n) <- c
            return (x, n `mappend` m)
\end{verbatim}
which is the multiplication in the monad \verb|WriterT w m|, modulo
some newtype contructor wrapping/unwrapping.

\begin{remark}
  There are other ways to explain why the functor $W = -\times M$ is
  part of a monad when $M$ is a monoid.  Let $\C$ be a cartesian
  category.  By currying the product functor $\times : \C \times \C
  \to \C$, we obtain a functor $\Gamma : \C \to [\C, \C]$, where $[\C,
  \C]$ denotes the category of functors from $\C$ to itself.  The
  functor $\Gamma$ maps an object $X\in\Ob\C$ to the functor $-\times
  X$ and a morphism $f$ to $f\times\id_X$.  The categories $\C$ and
  $[\C, \C]$ are naturally monoidal: in the former, the monoidal
  structure is given by cartesian product, and in the latter it is
  given by functor composition.  The functor $\Gamma$ is monoidal:
  $\Gamma(X\times Y) = -\times (X\times Y) \simeq (-\times Y) \circ
  (-\times X) = \Gamma Y \circ \Gamma X$.  The natural isomorphism is
  given by the associativity constraint of the monoidal structure
  induced by cartesian product.  Therefore, $\Gamma$ takes monoids in
  $\C$ to monoids in $[\C, \C]$.  The latter are precisely monads on
  the category $\C$.
\end{remark}

\begin{remark}
  The category $\D = M \textup{-} \Set$ introduced in this section is
  precisely the category of algebras over the monad $W = -\times M$,
  and the adjunction from $\C$ to $\D$ constructed above is an
  instance of the general construction of an adjunction from the
  category $\C$ to the category $\C^W$ of algebras over $W$.
\end{remark}

\begin{remark}
  Let $T$ be a strong monad on $\C$.  Suppose that $M$ is a monoid in
  $\C$.  The strength $t_{X, M} : TX\times M \to T (X\times M)$ is a
  natural transformation $t : (-\times M)\circ T \to T\circ (-\times
  M)$.  The functor $W = -\times M$ is a monad, and $t$ can be viewed
  as a \emph{distributive law} of the monad $T$ over $W$.  In fact,
  the four axioms of distributive laws reduce in this case to the four
  axioms of strengths.  This yields an alternative argument for why
  the composition of the monads $W$ and $T$ is again a monad.
  Furthermore, this also explains why the monad $T$ lifts to a monad
  $\bar{T}$ on the category $\C^W$ of algebras over the monad $W$:
  distributive laws $WT\to TW$ are in bijection with such liftings
  \cite{MR0241502}.
\end{remark}

\section{Reader monad transformer}

Although the considerations below can be carried out in any cartesian
closed category $\C$, we assume for the sake of simplicity that $\C$
is the category $\Set$ of sets.  Let $E$ be a set.  Define $\D$ to be
the category whose objects are the objects of $\C$, and for each pair
of objects $X$ and $Y$, the set of morphisms $\D(X, Y)$ is equal to
$\C(X\times E, Y)$.  We can think of $\D(X, Y)$ as the set of families
of functions from $X$ to $Y$ parametrized by $E$.  The identity
morphism of an object $X$ in $\D$ is the projection $\pr_1 : X\times E
\to X$.  Composition of morphisms $f: X\to Y$ and $g : Y\to Z$ in $\D$
(i.e., of maps $f : X \times E\to Y$ and $g : Y \times E\to Z$ in $\C$)
is given by
\[
g\ast f = \Bigl [X\times E\xrightarrow{\id_X\times\Delta} X\times (E\times E)
\xrightarrow{\alpha^{-1}_{X, E, E}} (X\times E)\times E
\xrightarrow{f\times\id_E} Y\times E\xrightarrow{g} X \Bigr].
\]
Here $\Delta : E\to E\times E$ is the diagonal map.  In other words,
$(g\ast f)(x, e) = g(f(x, e), e)$ for all $(x, e)\in X\times E$, but
it is helpful to have a diagrammatic representation of composition
that does not refer to elements.

There is an adjunction $(F, U, \eta, \eps)$ from $\C$ to $\D$.  The
functor $F : \C \to \D$ maps each object $X\in\Ob\C$ to itself and
each morphism $f : X\to Y$ to the composition $f\circ\pr_1 : X\times
E\to Y$.  In other words, to a function $f$ the functor $F$ assigns
the constant family of functions.  With this interpretation, $F$ is
clearly a functor.  The functor $U : \D \to \C$ maps an object
$X\in\Ob\D = \Ob\C$ to the exponential $X^E$, and a morphism $f : X\to
Y$ in $\D$ (i.e., a morphism $f : X\times E\to Y$ in $\C$) to the
morphisms corresponding to the composite
\[
f\ast\ev = \Bigl[ X^E\times E\xrightarrow{\id_{X^E}\times\Delta}
X^E\times (E\times E) \xrightarrow{\alpha^{-1}_{X^E, E, E}}(X^E\times
E)\times E \xrightarrow{\ev\times\id_E} X\times E\xrightarrow{f} Y
\Bigr]
\]
by the closedness of the category $\C$.  Here $\ev : X^E\times E\to X$
is the evaluation morphism.  Because we are assuming that $\C$ is the
category of sets, $U(f)$ can be written as $\lambda g.\; \lambda e.\;
f(g(e), e)$.  Informally, an $E$-indexed family of functions $\{f_e :
X\to Y\}_{e\in E}$ is mapped to the function that takes an $E$-indexed
family of elements $\{x_e\}_{e\in E}$ of the set $X$ as input and
applies each function to the corresponding element, producing a new
$E$-indexed family of elements $\{f_e(x_e)\}_{e\in E}$ of elements of
the set $Y$.  This makes it obvious that $U$ is a functor.  The unit
$\eta_X : X \to FUX = X^E$ is given by $\eta_X(x) = \lambda e.\; x$.
The counit $\eps_X : FUX \to X$ is a morphism in $\D$ represented by
the evaluation morphism $\ev : X^E\times E\to X$ in $\C$.  The monad
associated with this adjunction is precisely the reader monad with the
environment $E$.

Let $(T, e^T, m^T, t)$ be a strong monad on the category $\C$.  Like
in the case of writer monad transformer, we would like to lift $T$ to
a monad $\bar{T}$ on the category $\D$, which we then could translate
back to $\C$ along the adjunction $(F, U, \eta, \eps)$.  The functor
$T$ is lifted to the category $\D$ as follows: $\bar{T}X = TX$ and for
each morphism $f\in\D(X, Y) = \C(X\times E, Y)$, we set
\[
\bar{T}(f) = \Bigl[ TX\times E\xrightarrow{t_{X, E}} T(X\times E)
\xrightarrow{T(f)} TY \Bigr].
\]
Let us check that $\bar{T}$ preserves composition and identities.  Let
$f\in\D(X, Y)=\C(X\times E, Y)$ and $g\in\D(Y, Z)=\C(Y\times E, Z)$.
The equation $\bar{T}(g\ast f) = \bar{T}(g)\ast \bar{T}(f)$
coincides with the exterior of the diagram:
\[
\begin{xy}
  (-22, 36)*+{TX\times E}="1";
  (22, 36)*+{T(X\times E)}="2";
  (-22, 18)*+{TX\times (E\times E)}="3";
  (22, 18)*+{T(X\times (E\times E))}="4";
  (-22, 0)*+{(TX\times E)\times E}="5";
  (-22, -18)*+{T(X\times E)\times E}="6";
  (22, -18)*+{T((X\times E)\times E)}="7";
  (-22, -36)*+{TY\times E}="8";
  (22, -36)*+{T(Y\times E)}="9";
  (44, -36)*+{TZ}="10";
  {\ar@{->}^-{t_{X,E}} "1";"2"};
  {\ar@{->}_-{\id_{TX}\times\Delta} "1";"3"};
  {\ar@{->}^-{T(\id_X\times\Delta)} "2";"4"};
  {\ar@{->}^-{t_{X, E\times E}} "3";"4"};
  {\ar@{->}_-{\alpha^{-1}_{TX, E, E}} "3";"5"};
  {\ar@{->}^-{T(\alpha^{-1}_{X,E,E})} "4";"7"};
  {\ar@{->}_-{t_{X,E}\times\id_E} "5";"6"};
  {\ar@{->}^-{t_{X\times E, E}} "6";"7"};
  {\ar@{->}_-{T(f)\times\id_E} "6";"8"};
  {\ar@{->}^-{T(f\times\id_E)} "7";"9"};
  {\ar@{->}^-{t_{Y,E}} "8";"9"};
  {\ar@{->}^-{T(g)} "9";"10"};
\end{xy}
\]
The squares commute by the naturality of $t$.  The pentagon is, up to
the orientation of the associativity isomorphism, one of the strength
axioms.  Preservation of identities follows from the diagram:
\[
\begin{xy}
  (-30, 18)*+{TX\times E}="1";
  (0, 18)*+{T(X\times E)}="2";
  (-30, 0)*+{TX\times\unito}="3";
  (0, 0)*+{T(X\times\unito)}="4";
  (30, 0)*+{TX}="5";
  {\ar@{->}^-{t_{X,E}} "1";"2"};
  {\ar@{->}_-{\id_{TX}\times !_E} "1";"3"};
  {\ar@{->}_-{T(\id_X\times !_E)} "2";"4"};
  {\ar@{->}^-{t_{X,\unito}} "3";"4"};
  {\ar@{->}^-{T(\pr_1)} "4";"5"};
  {\ar@{->}^-{T(\pr_1)} "2";"5"};
  {\ar@{->}@/_2pc/_{\pr_1} "3";"5"};
\end{xy}
\]
The square commutes by the naturality of $t$.  The commutativity of
the right triangle follows from the obvious identity
$\pr_1\circ(\id_X\times !_E) = \pr_1$ and functoriality of $T$.  The
bottom triangle is one of the strength axioms (note that $\rho^{-1}_X
= \pr_1 : X\times\unito\to X$).  The left-bottom composite is equal to
$\pr_1 : TX\times E\to TX$, which represents the identity morphism of
$TX$ in the category $\D$.

We have shown that $\bar{T}$ is a functor from the category $\D$ to
itself.  Let us check that the families of morphisms
\begin{align*}
  e^{\bar{T}}_X & = e^T_X\circ\pr_1 = F(e^T_X) \in \C(X\times E, TX) =
  \D(X, \bar{T}X),
  \\
  m^{\bar{T}}_X & = m^T_X\circ\pr_1 = F(m^T_X) \in \C(TTX\times E, TX)
  = \D(\bar{T}\bar{T}X, \bar{T}X)
\end{align*}
are natural transformations $\Id_{\D}\to\bar{T}$ and
$\bar{T}\bar{T}\to\bar{T}$.  We only give a proof for the second
family.  Let $f\in\D(X, Y)=\C(X\times E, Y)$.  We have to show that
the diagram
\[
\begin{xy}
  (-15, 9)*+{\bar{T}\bar{T}X}="1";
  (15, 9)*+{\bar{T}\bar{T}Y}="2";
  (-15, -9)*+{\bar{T}X}="3";
  (15, -9)*+{\bar{T}Y}="4";
  {\ar@{->}^-{\bar{T}\bar{T}(f)} "1";"2"};
  {\ar@{->}_-{m^{\bar T}_X} "1";"3"};
  {\ar@{->}^-{m^{\bar T}_Y} "2";"4"};
  {\ar@{->}^-{\bar{T}(f)} "3";"4"};
\end{xy}
\]
commutes in $\D$.  Using the definition of composition in $\D$ and the
definition of $\bar{T}$, it is not difficult to convince yourself that
this diagram coincides with the exterior of the following diagram:
\[
\begin{xy}
  (-75, 36)*+{TTX\times E}="1";
  (-75, 18)*+{TTX\times (E\times E)}="2";
  (-75, 0)*+{(TTX\times E)\times E}="3";
  (-25, 0)*+{T(TX\times E)\times E}="4";
  (25, 0)*+{TT(X\times E)\times E}="5";
  (70, 0)*+{TTY\times E}="6";
  (-75, -18)*+{TTX\times E}="7";
  (-25, -18)*+{T(TX\times E)}="8";
  (25, -18)*+{TT(X\times E)}="9";
  (70, -18)*+{TTY}="10";
  (-75, -36)*+{TX\times E}="11";
  (25, -36)*+{T(X\times E)}="12";
  (70, -36)*+{TY}="13";
  {\ar@{->}_-{\id_{TTX}\times\Delta} "1";"2"};
  {\ar@{->}_-{\alpha^{-1}_{TTX,E,E}} "2";"3"};
  {\ar@{->}^-{t_{TX,E}\times\id_E} "3";"4"};
  {\ar@{->}^-{T(t_{X,E})\times\id_E} "4";"5"};
  {\ar@{->}^-{TT(f)\times\id_E} "5";"6"};
  {\ar@{->}_-{\pr_1\times\id_E} "3";"7"};
  {\ar@{->}_-{\pr_1} "4";"8"};
  {\ar@{->}_-{\pr_1} "5";"9"};
  {\ar@{->}^-{\pr_1} "6";"10"};
  {\ar@{->}^-{t_{TX,E}} "7";"8"};
  {\ar@{->}^-{T(t_{X,E})} "8";"9"};
  {\ar@{->}^-{TT(f)} "9";"10"};
  {\ar@{->}_-{m^T_X\times\id_E} "7";"11"};
  {\ar@{->}_-{m^T_{X\times E}} "9";"12"};
  {\ar@{->}^-{t^T_Y} "10";"13"};
  {\ar@{->}^-{t_{X,E}} "11";"12"};
  {\ar@{->}^-{T(f)} "12";"13"};
\end{xy}
\]
The three top squares commute for trivial reasons.  The remaining
square commutes by the naturality of $m^T$.  The pentagon is one of
the strength axioms.

We have shown that $e^{\bar{T}}_X = F(e^T_X)$ and $m^{\bar{T}}_X =
F(m^T_X)$ give rise to natural transformations $\Id_{\D}\to\bar{T}$
and $\bar{T}\bar{T}\to\bar{T}$, respectively.  That they satisfy the
monad laws follows from the fact that $e^T$ and $m^T$ satisfy the
monad laws, and from the functoriality of $F$.

Thus, we have lifted the monad $(T, e^T, m^T)$ on the category $\C$ to
the monad $(\bar{T}, e^{\bar T}, m^{\bar T})$ on the category $\D$,
which we can now translate along the adjunction $(F, U, \eta, \eps)$
back to the category $\C$.  The composite monad $(P, e^P, m^P)$ is
given by $PX = (TX)^E$,
\begin{align*}
  e^P_X(x) & = (\lambda g.\;\lambda e.\; e^T_X(g(e)))(\lambda e.\; x) =
  \lambda e. e^T_X(x),\\
  m^P_X & = (\lambda g.\; \lambda e.\; m^T_X(g(e)))\circ (\lambda h.\;
  \lambda e.\; T(\ev) (t_{(TX)^E, E}(h(e), e))),
\end{align*}
where $\ev: (TX)^E\times E\to TX$ is the evaluation map.
Eta-expanding the definition of $m^P_X$, we obtain:
\[
m^P_X(h) = \lambda e.\; m^T_X(T(\ev)(t_{(TX)^E, E}(h(e), e))).
\]
Translating this into Haskell notation, we obtain:
\begin{verbatim}
join h = \e -> join $ fmap ev $ strength (h e, e)
    where
      ev (f, v)       = f v
      strength (c, y) = c >>= \x -> return (x, y)
\end{verbatim}
which can be transformed as follows:
\begin{verbatim}
join h = \e -> join $ fmap ev $ strength (h e, e)

       -- definitions of `join' and `fmap'
       = \e -> (strength (h e, e) >>= (return . ev)) >>= id

       -- associativity axiom for monads
       = \e -> strength (h e, e) >>= \(f, v) -> return (ev (f, v)) >>= id

       -- left identity axiom for monads
       = \e -> strength (h e, e) >>= \(f, v) -> ev (f, v)

       -- definition of `ev'
       = \e -> strength (h e, e) >>= \(f, v) -> f v

       -- definition of `strength'
       = \e -> (h e >>= \x -> return (x, e)) >>= \(f, v) -> f v

       -- associativity axiom for monads
       = \e -> h e >>= \x -> return (x, e) >>= \(f, v) -> f v

       -- left identity axiom for monads
       = \e -> h e >>= \f -> f e

       -- syntactic sugar
       = \e -> do f <- h e
                  f e
\end{verbatim}
Modulo newtype constructor wrapping/unwrapping, this is precisely the
multiplication in the monad \verb|ReaderT e m|.

\begin{remark}
  Here is another, higher-level explanation of why the functor $R =
  (-)^E$ is part of a monad.  Let $\C$ be a cartesian closed category.
  Currying the internal hom functor $(-)^- : \C^{\op}\times \C\to \C$,
  $(X, Y) \mapsto Y^X$, we obtain a functor $\Gamma : \C^{\op}\to [\C,
  \C]$, $X\mapsto (-)^X$.  The functor $\Gamma$ is monoidal:
  $\Gamma(X\times Y) = (-)^{X\times Y} \simeq (-)^X\circ (-)^Y =
  \Gamma X\circ \Gamma Y$.  Any object $E$ of the category $\C$ is
  naturally a coalgebra (more appropriately called `comonoid') in
  $\C$: the comultiplication $\Delta : E\to E\times E$ is the diagonal
  morphism (a unique morphism such that $\pr_1\circ\Delta =
  \pr_2\circ\Delta = \id_E$), and the counit $\eps : E\to \unito$ is
  the final morphism $!_E$.  Therefore, $E$ is a monoid in the
  opposite category $\C^{\op}$.  The monoidal functor $\Gamma$ takes
  monoids in $\C^{\op}$ to monoids in $[\C,\C]$.  The latter are
  precisely monads on the category $\C$.  Therefore, $\Gamma E$ has
  the structure of a monad.
\end{remark}

\begin{remark}
  The category $\D$ introduced at the beginning of this section is
  isomorphic to the Kleisli category $\C_R$ of the monad $R$: $\Ob\C_R
  = \Ob\C$, $\C_R(X, Y) = \C(X, RY) = \C(X, Y^E)$, the identity of an
  object $X$ is the unit $e^R_X : X\to X^E$ of the monad, and
  composition is given by $g\circ_R f = m^R_Z \circ R(g)\circ f$, for
  all $f\in\C_R(X, Y) = \C(X, Y^E)$, $g\in\C_R(Y, Z) = \C(Y, Z^E)$.
  The isomorphism between $\D$ and $\C_R$ is given by the identity map
  on objects and the closedness bijection $\C(X\times E, Y) \simeq
  \C(X, Y^E)$ on morphisms.  The adjunction from $\C$ to $\D$
  introduced above is an instance of a general construction, the
  adjunction from the category $\C$ to the Kleisli category of the
  monad $R$, translated along this isomorphism of categories.
\end{remark}

\section{Error monad transformer}

Let $\C$ be the category of sets.  Let $E$ be a set.  Consider the
category $\D = E/\C$, the under category (also known as coslice
category) of the object $E$: objects are maps $\varphi : E\to X$ in
$\C$, and a morphism from $\varphi : E \to X$ to $\psi : E \to Y$ is a
map $f : X\to Y$ such that $f\circ\varphi = \psi$.  Composition and
identities are the obvious ones.  There is an adjunction $(F, U, \eta,
\eps)$ from $\C$ to $\D$.  The functor $F$ maps a set $X$ to the
morphism $\inl : E \to E+X$ and a map $f$ to $\id_E + f$.  Here $+$
denotes disjoint union of sets.  The functor $U : \D \to \C$ is the
forgetful functor, mapping a function $E\to X$ to its codomain $X$.
The unit $\eta_X : X\to E+X$ is the map $\inr$, and the counit is the
morphism $\eps_{\varphi : E\to X}: (E\xrightarrow{\inl} E+X) \to
(E\xrightarrow{\varphi} X)$ given by $\varphi \vee \id_X$.  Here for a
pair of maps $f : X\to Z$ and $g : Y\to Z$, $f\vee g$ denotes the
unique map $X+Y\to Z$ such that $(f\vee g)\circ\inl = f$ and $(f\vee
g)\circ\inr = g$.  The monad associated with this adjunction is
precisely the error monad with the set of errors $E$.

Let $(T, e^T, m^T)$ be a monad on the category $\C=\Set$.  First of
all, the functor $T$ lifts to a functor $\bar{T}$ on the category $\D
= E/\Set$: $\bar{T}(\varphi : E\to X) = T(\varphi)\circ e^T_E$, and
$\bar{T}(f) = T(f)$ for each $f \in \D(\varphi : E \to X, \psi : E\to
Y) \subset \C(X, Y)$.  Clearly, $\bar{T}$ preserves identities and
composition.  Furthermore, the families of morphisms $e^T_X : X\to TX$
and $m^T_X : TTX \to TX$ can also be viewed as natural transformations
$e^{\bar T}_{\varphi : E\to X} : (E\xrightarrow{\varphi} X)\to
(E\xrightarrow{T(\varphi)\circ e^T_E}TX)$ and
\[
m^{\bar T}_{\varphi : E\to X} : (E\xrightarrow{TT(\varphi)\circ
  T(e^T_E)\circ e^T_E}TTX)\to (E\xrightarrow{T(\varphi)\circ e^T_E} TX).
\]
That $e^{\bar T}$ is a morphism in the category $\D$ is a consequence
of the naturality of $e^T$, and that $m^{\bar T}$ is a morphism in
$\D$ follows from the naturality of $m^T$ and the right identity axiom
for monads.  The naturality of $e^{\bar T}$ and $m^{\bar T}$ follows
from the naturality of $e^T$ and $m^T$.  Hence, $(\bar{T}, e^{\bar T},
m^{\bar T})$ is a monad on the category $\D$, which we can now
translate back to the category $\C$.  The composite monad $(P, e^P,
m^P)$ is given by $PX = T(E+X)$, $e^P_X = e^T_{E+X}\circ\inr$, $m^P_X
= m^T_{E+X}\circ T((T(\inl)\circ e^T_E)\vee\id_{T(E+X)})$.  Translating
the last equation into Haskell notation, we obtain:
\begin{verbatim}
join = join . fmap eps
    where
      eps = either (fmap Left . return) id
\end{verbatim}
Eta-expanding and transforming this equation, we obtain:
\begin{verbatim}
join z = join (fmap eps z)

       -- definitions of `join' and `fmap'
       = (z >>= (return . eps)) >>= id

       -- axioms of monads (see above)
       = z >>= eps

       -- syntactic sugar
       = do y <- z
            eps y

       -- definition of `either'
       = do y <- z
            case y of
              Left e  -> (fmap Left . return) e
              Right x -> id x

       -- naturality of `return': fmap Left . return = return . Left
       = do y <- z
            case y of
              Left e  -> return (Left e)
              Right x -> x
\end{verbatim}
Modulo newtype constructor wrapping/unwrapping this is precisely the
multiplication in \verb|ErrorT e m|.

\bibliographystyle{plain}
\bibliography{mathscinet,calc-mts-with-cat-th}

\begin{thebibliography}{1}

\bibitem{MR0241502}
Jon Beck.
\newblock Distributive laws.
\newblock In B.~Eckmann, editor, {\em Seminar on Triples and Categorical
  Homology Theory}, pages 119--140, Berlin, Heidelberg, 1969. Springer Berlin
  Heidelberg.

\bibitem{calc-mon-cat-th}
Derek Elkins.
\newblock Calculating monads with category theory.
\newblock {\em The Monad.Reader}, 13, 2009.

\bibitem{MR1712872}
Saunders Mac~Lane.
\newblock {\em Categories for the working mathematician}.
\newblock Springer-Verlag, New York-Berlin, 1971.
\newblock Graduate Texts in Mathematics, Vol. 5.

\bibitem{MR1115262}
Eugenio Moggi.
\newblock Notions of computation and monads.
\newblock {\em Information and Computation}, 93(1):55--92, 1991.
\newblock Selections from 1989 IEEE Symposium on Logic in Computer Science.

\end{thebibliography}

\end{document}